\documentclass[12pt]{amsart}

\usepackage{amsmath,amssymb,amsbsy,amsfonts,amsthm,latexsym,
                     amsopn,amstext,amsxtra,euscript,amscd,mathrsfs}
\usepackage[colorlinks,linkcolor=blue,anchorcolor=blue,citecolor=blue]{hyperref}

\newtheorem{thm}{Theorem}

\newtheorem{lemma}[thm]{Lemma}

\newtheorem{theorem}[thm]{Theorem}
\newtheorem{cor}[thm]{Corollary}
\newtheorem{corollary}[thm]{Corollary}

\numberwithin{equation}{section}
\numberwithin{thm}{section}

\def\cF{{\mathcal F}}

\def\cI{{\mathcal I}}
\def\cJ{{\mathcal J}}

\def\cQ{{\mathcal Q}}
\def\cR{{\mathcal R}}
\def\cS{{\mathcal S}}
\def\cT{{\mathcal T}}

\def\cW{{\mathcal W}}

\def\cY{{\mathcal Y}}
\def\cZ{{\mathcal Z}}

\def\F{{\mathbb F}}

\def\K{{\mathbb K}}

\def\Q{{\mathbb Q}}

\def\Z{{\mathbb Z}}

\def\sN{{\mathscr N}}

\def\ep{{\mathbf{e}}_p}

\def\em{{\mathbf{e}}_m}

\def\\{\cr}
\def\({\left(}
\def\){\right)}
\def\[{\left[}
\def\]{\right]}
\def\<{\langle}
\def\>{\rangle}

\def\eps{\varepsilon}
\def\mand{\qquad\mbox{and}\qquad}

\def\GL{\operatorname{GL}}

\def\sym{\operatorname{sym}}

\begin{document}

\title[The Sato--Tate Conjecture In  Families of Curves]{The Sato--Tate Distribution in Families of Elliptic Curves with a Rational Parameter of Bounded Height}

\author{Min Sha}
\address{School of Mathematics and Statistics, University of New South Wales,
 Sydney NSW 2052, Australia}
\email{shamin2010@gmail.com}

\author{Igor E. Shparlinski}
\address{School of Mathematics and Statistics, University of New South Wales,
 Sydney NSW 2052, Australia}
\email{igor.shparlinski@unsw.edu.au}

\date{}

\begin{abstract}
We obtain new results concerning the  Sato--Tate conjecture on
the distribution of Frobenius angles
over parametric families of elliptic curves  with a rational parameter of bounded height.
\end{abstract}

\maketitle

\section{Introduction}
\label{sec:intro}

\subsection{Background and motivation}

Let  $f(Z), g(Z) \in \Z[Z]$  be  polynomials satisfying
\begin{equation}
\label{eq:Nondeg}
\Delta(Z) \neq 0 \mand j(Z)  \not\in \Q,
\end{equation}
where
$$
\Delta(Z)   = -16 (4f(Z)^3 + 27g(Z)^2)\quad\text{and}\quad
j(Z)    = \frac{-1728(4f(Z))^{3}}{\Delta(Z)}.
$$
Then we consider
the   elliptic curve
\begin{equation}
\label{eq:Family AB}
E(Z) : \quad Y^2 = X^3 + f(Z)X + g(Z)
\end{equation}
over the function field $\Q(Z)$. Thus, $\Delta(Z)$ and
$j(Z)$ are the {\it discriminant\/} and {\it $j$-invariant\/} of this elliptic curve, respectively;
see~\cite{Silv} for a general background on elliptic curves.

In  what follows, we also refer
to~\cite{Silv}  for the definition of the {\it conductor\/} $N_E$ of an elliptic curve $E$
as well as for the notions of  {\it CM curves\/} and {\it non-CM curves}.

The properties of
the specialisations $E(t)$ modulo consecutive primes $p\le x$
for a growing parameter $x$ and for  the parameter $t$ that
runs through some interesting sets $\cT$ have recently being
investigated quite intensively,  see~\cite{Coj,ShaShp,Shp4} and
also Section~\ref{sec:prev res}.
These sets $\cT$ can be of integer or rational
numbers of limited size, and sometimes also of certain arithmetic
structure; for example $\cT$ can be a set of primes in a given
interval $[1,T]$, see~\cite{dlBSSV}.

Throughout the paper, for an elliptic curve $E$ over $\Q$ and a prime $p \nmid N_E$
we   denote by $E_p$ the reduction of $E$ modulo $p$, which is an elliptic curve defined over
the finite field $\F_p$ of
$p$ elements.
Furthermore, we use
$E_p(\F_p)$ to denote the group of $\F_p$-rational
points on $E_p$.  In particular, $a_p(E)=p+1-\#E_p(\F_p)$ is
the {\it Frobenius trace\/}.

From the Hasse bound (see~\cite{Silv}):
$|a_p(E)| \le 2 \sqrt{p}$, we can define the {\it Frobenius angle\/} $\psi_p(E) \in [0, \pi]$ by the
equation
\begin{equation}
\label{eq:ST angle}
\cos \psi_p(E) = \frac{a_p(E)}{2\sqrt{p}}.
\end{equation}
Then, the \textit{Sato--Tate  conjecture} predicts that  the angles $\psi_p(E)$
are distributed in $[0,\pi]$ with the
 {\it Sato--Tate  density\/}
\begin{equation}
\label{eq:ST dens}
\mu_{\tt ST}(\alpha,\beta) = \frac{2}{\pi}\int_\alpha^\beta
\sin^2\vartheta\, d \vartheta = \frac{2}{\pi}\int_{\cos \beta}^{\cos \alpha}
(1-z^2)^{1/2}\, d z,
\end{equation}
where $[\alpha,\beta] \subseteq [0,\pi]$.

 The Sato--Tate conjecture
has  been settled only quite recently  in the series of works of Barnet-Lamb,  Geraghty,
Harris and Taylor~\cite{B-LGHT},
Clozel,  Harris and Taylor~\cite{CHT},
Harris,  Shepherd-Barron and  Taylor~\cite{HS-BT}, and Taylor~\cite{Taylor2008}.
In particular, given a non-CM elliptic curve $E$ over $\Q$ of conductor $N_E$,
for the number $\pi_{E}(\alpha,\beta;x)$ of
primes $p \le x$ with $p \nmid N_E$
for which
$\psi_p(E) \in[\alpha, \beta] \subseteq [0,\pi]$, we have
$$
\pi_{E}(\alpha,\beta;x) \sim
\mu_{\tt ST}(\alpha,\beta) \cdot\frac{x}{\log x}
$$
as $x \to \infty$.

However, the above asymptotic formula is lack of an explicit error term. So, it makes sense
to study $\pi_{E}(\alpha,\beta;x)$ on average over
some natural families of elliptic curves.
In this paper, we continue  this line of research and in particular
introduce new families of curves with a rational parameter.

\subsection{Previous  results}
\label{sec:prev res}

As one of the possible relaxation of the still open  \textit{Lang--Trotter conjecture}, see~\cite{Lang},
 Fouvry and  Murty~\cite{FoMu} have introduced the study
of  the  reductions $E_p$ for $p\le x$
on average  over a family of elliptic curves $E$.
More precisely, in~\cite{FoMu}   the frequency of vanishing $a_p(E_{u,v}) = 0$ is
investigated for  the family of curves
\begin{equation}
\label{eq:Family uv}
E_{u,v}:\ Y^2 = X^3 + uX + v,
\end{equation}
with  the integer parameters $(u, v) \in [-U,U]\times [-V,V]$.
This has been extended
to arbitrary  values $a_p(E_{u,v}) = a$ by
David and Pappalardi~\cite{DavPapp} and more recently by
Baier~\cite{Baier1}, see also~\cite{Baier2}.

The approach of~\cite{Baier1,Baier2,DavPapp,FoMu}
also applies to the Sato--Tate conjecture on average
for the family~\eqref{eq:Family uv}, see~\cite{BaZha},
provided that $U$ and $V$ are reasonably large compared to $p$.
Banks and Shparlinski~\cite{BaSh} have shown that using
a different approach, based on
bounds of multiplicative character sums and the large sieve
inequality (instead of the exponential sum technique employed in~\cite{FoMu}),
one can study ``thinner'' families, that is, establish the Sato--Tate conjecture
on average for the curves~\eqref{eq:Family uv} for smaller values of $U$ and $V$.

The technique of~\cite{BaSh} has been used in several other
problems, see~\cite{CojDav,DavJ-U,Shp2,Shp3}.
In particular,  the  Sato--Tate conjecture has been established on average for
several other families of curves. For example,
Shparlinski~\cite{Shp3} has  studied the family of elliptic curves $Y^2 = X^3 + f(u)X +g(v)$
with  integers $|u| \le U$, $|v|\le V$, where $f,g \in \Z[Z]$.

Sha and Shparlinski~\cite{ShaShp} have established the Sato--Tate conjecture on average
for the families of curves $Y^2 = X^3 + f(u+v)X +g(u+v)$, where $u,v$ both run through some subsets of $\{1,2,\ldots,T\}$, or both run over the set
\begin{equation}
\label{set:Farey}
\cF(T) = \{u/v \in \Q ~:~ \gcd(u, v) = 1, \ 1 \le
u,v \le T\}.
\end{equation}
  For the size of $\cF(T)$, it is well known that
\begin{equation}
\label{eq:Farey}
\# \cF(T)  \sim \frac{6}{\pi^2}T^2,
\end{equation}
as $T\to \infty$, see~\cite[Theorem~331]{HW}.
We recall that the set $\cF(T)\cap[0,1]$ is the well-known
set of {\it Farey fractions\/}.  We note that all the related results of~\cite{ShaShp} hold
without any changes if one replaces the set $\cF(T)$ with $\cF(T)\cap[0,1]$.

In addition,  for the family
of curves~\eqref{eq:Family AB}, Cojocaru and Hall~\cite{CojHal} have
given  an upper bound on the
frequency of the event $a_p(E(t)) = a$ for a fixed integer $a$,
when the parameter $t$ runs through the set $\cF(T)$.
This bound has been improved by Cojocaru and Shparlinski~\cite{CojShp}
and then further improved by Sha  and Shparlinski~\cite{ShaShp}.

Most recently, de la Bret\`eche, Sha, Shparlinski and Voloch~\cite{dlBSSV}   have established the  Sato--Tate conjecture on average for the polynomial family~\eqref{eq:Family AB} of elliptic curves when the variable $Z$
is specialised to a parameter $t$ from sets of prescribed
multiplicative structure, such as prime numbers, and geometric progressions. Particularly, the Sato--Tate conjecture on average is true for the families of curves $Y^2 = X^3 + f(uv)X +g(uv)$, where $u,v$ both run through some subsets of $\{1,2,\ldots,T\}$.

\subsection{General notation}
\label{sec:notation}
As usual the expressions $A=O(B)$ and $A\ll B$ (sometimes we will write this also as $B \gg A$) are both equivalent to the inequality $|A|\le cB$ with some absolute constant $c>0$, $A=o(B)$ means that $A/B\to 0$ and $A \sim B$  means that $A/B\to 1$.
We also write $A\asymp B$ if $A \ll B \ll A$.

Throughout the paper  the implied constants  may, where obvious, depend on the polynomials $f$ and $g$ in~\eqref{eq:Family AB}  and the real positive parameter $\eps$,
and are absolute otherwise.

Furthermore,  the letter $p$ always denotes a prime number.
We always assume that the elements of $\F_p$
are represented by the set $\{0, \ldots, p-1\}$ and thus
we switch freely between the equations in $\F_p$ and congruences
modulo $p$.

As usual, we use $\pi(x)$ to denote the number of primes  $p\le x$.

For a subset $\cS$ in the real plane, we denote by $\sN(\cS) = \#(\cS\cap \Z^2)$ the  number of integral lattice points in $\cS$.

\section{Main Results}

In this paper, we establish the Sato--Tate conjecture on average for some families of elliptic curves with a rational parameter.

Recall that for any $t\in \Q$ with $\Delta(t)\ne 0$,
we use $\pi_{E(t)}(\alpha,\beta; x)$
to denote the number of primes $p \le x$ with $p \nmid N_{E(t)}$ (or equivalently, $\Delta(t) \not\equiv 0 \pmod p$, see Section~\ref{sec:not}) and
$\psi_p(E(t))\in[\alpha,\beta]$.

We start with a general result.  Let
\begin{equation}
\label{eq:interv}
\cI(A,T)  = [A + 1, A + T],\ \cJ(B,T) = [B + 1, B + T],
\end{equation}
be  two intervals  with non-negative integers  $A,B$ and positive integer $T$.
\begin{theorem}
\label{thm:S-T conv}
Suppose that the polynomials $f(Z), g(Z) \in \Z[Z]$
satisfy~\eqref{eq:Nondeg}.
 Let $\cI(A,T)$ and $\cJ(B,T)$ be
two intervals of the form~\eqref{eq:interv}.
Let $\cW\subseteq \cI(A,T)\times\cJ(B,T)$
be an arbitrary convex subset.
Then, uniformly over  $[\alpha,\beta] \subseteq [0,\pi]$, we have
\begin{equation*}
\begin{split}
 \frac{1}{ \pi(x) \sN(\cW)}
 \sum_{\substack{(u,v) \in \cW \cap \Z^2 \\ \Delta(u/v) \ne 0}}
\pi_{E(u/v)}(\alpha&,\beta; x) - \mu_{\tt ST}(\alpha,\beta) \\
&\ll \frac{T \log x}{x}   + \frac{T}{\sN(\cW)}  + \frac{T^{1/4}x^{1/2+o(1)}}{\sN(\cW) ^{1/2}}.
 \end{split}
\end{equation*}
\end{theorem}

In particular,  we have:

\begin{cor}
\label{cor:S-T conv}
Suppose that the polynomials $f(Z), g(Z) \in \Z[Z]$
satisfy~\eqref{eq:Nondeg}.
Let $\cW\subseteq \cI(A,T)\times\cJ(B,T)$
be an arbitrary convex subset such that
\begin{equation}
\label{eq:condition}
\sN(\cW) \gg T^\eta
\end{equation}
for some real $\eta > 3/2$. Assume that for sufficiently small $\eps>0$,
$$
 x^{2/(2\eta-1)+\eps} \le T \le x^{1 - \eps}.
$$
Then, uniformly over  $[\alpha,\beta] \subseteq [0,\pi]$, we have
$$
 \frac{1}{ \pi(x) \sN(\cW)}
 \sum_{\substack{(u,v) \in \cW \cap \Z^2 \\ \Delta(u/v) \ne 0}}
\pi_{E(u/v)}(\alpha,\beta; x)
 = \mu_{\tt ST}(\alpha,\beta) +O\(x^{-\eps/2 + o(1)}  \).
$$
\end{cor}

Note that the fact $\sN(\cW) \ll T^2$ and the condition~\eqref{eq:condition}
imply that we also have $\eta \le 2$ in Corollary~\ref{cor:S-T conv}.

We now  give a  natural class of subsets
that meet the condition~\eqref{eq:condition}
in Corollary~\ref{cor:S-T conv}. Namely, by the {\it Pick's theorem\/}
(see~\cite[Theorem~2.8]{BeRob})
this holds for any convex simple polygon with vertices on the integral lattice $\Z^2$ whose area is not less than $T^\eta$ up to a constant.

Assume that we further have
$$
\# \{(u,v)\in \cW \cap \Z^2~:~\gcd(u,v)=1 \} \asymp  \sN(\cW).
$$
Then, one can similarly establish the Sato-Tate conjecture on average
for $\pi_{E(u/v)}(\alpha,\beta; x)$, where $(u,v)\in \cW \cap \Z^2$ with $\gcd(u,v)=1$, as in Corollary~\ref{cor:S-T conv}.

If we choose $\cW= \cI(A,T)\times\cJ(B,T)$, then we can take  $\eta=2$ and
$x^{2/3+\eps} \le T \le x^{1 - \eps}$  in Corollary~\ref{cor:S-T conv}.
In the following we want to relax this condition on $T$ in the case when $A,B \le T$.

We now define the sets:
 \begin{equation}
\label{eq:set Z}
\cZ(A,B,T)=\( \cI(A,T)\times \cJ(B,T) \)\cap \Z^2,
\end{equation}
and
 \begin{equation}
\label{eq:set Z*}
\cZ^*(A,B,T)= \{(u,v)\in \cZ(A,B,T)~:~\gcd(u,v)=1 \}.
\end{equation}

\begin{theorem}
\label{thm:S-T Farey}
Suppose that the polynomials $f(Z), g(Z) \in \Z[Z]$
satisfy~\eqref{eq:Nondeg}.
 Let  $A,B \le T$ and let  $\cY$ be one of the sets $\cZ(A,B,T)$
 or $\cZ^*(A,B,T)$.
 Then, uniformly over  $[\alpha,\beta] \subseteq [0,\pi]$, we have
\begin{equation*}
\begin{split}
 \frac{1}{ \pi(x) \#\cY}  \sum_{\substack{(u,v) \in \cY \\ \Delta(u/v) \ne 0}}
\pi_{E(u/v)}(\alpha,\beta&; x)
  - \mu_{\tt ST}(\alpha,\beta) \\
& \ll
x^{-1/4} + T^{-1/2+o(1)}x^{1/4}.
 \end{split}
\end{equation*}
\end{theorem}

We remark that $\#\cZ(A,B,T)=(T+1)^2$ and
$$
\#\cZ^*(A,B,T) \asymp T^2.
$$
If $T \ge x^{1/2+\eps}$ with some positive $\varepsilon \le 1/2$, then the error term in Theorem~\ref{thm:S-T Farey} becomes $O\(x^{-\eps/2 + o(1)}  \)$.
Note that the set $\cF(T)$ defined in~\eqref{set:Farey} is exactly the following set
$$
  \{u/v~:~(u,v)\in  \cZ^*(0,0,T) \}.
$$
Thus, we have:
\begin{corollary}
\label{cor:S-T Farey}
Suppose that the polynomials $f(Z), g(Z) \in \Z[Z]$
satisfy~\eqref{eq:Nondeg}, and
 for some sufficiently small  $\eps > 0$ we have
$$
  T \ge x^{1/2+\eps}.
$$
Then, uniformly over  $[\alpha,\beta] \subseteq [0,\pi]$, we have
$$
 \frac{1}{ \pi(x) \#\cF(T)}\sum_{\substack{s \in \cF(T) \\ \Delta(s) \ne 0}}
\pi_{E(s)}(\alpha,\beta; x)
 = \mu_{\tt ST}(\alpha,\beta) +O\(x^{-\eps/2 + o(1)}  \).
$$
\end{corollary}

One can also consider the Sato-Tate conjecture on average with the
product $uv$ for $u,v \in  \cZ(A,B,T)$.

Here, we  present  the following result for the family of elliptic curves parameterized by products $rs$ with $r , s$ from arbitrary subsets of  $\cF(T)$.

\begin{theorem}
\label{thm:S-T Farey Prod}
Suppose that the polynomials $f(Z), g(Z) \in \Z[Z]$
satisfy~\eqref{eq:Nondeg}.
Then, for any subsets $\cR, \cS \subseteq \cF(T)$, uniformly over  $[\alpha,\beta] \subseteq [0,\pi]$ we have
\begin{equation*}
\begin{split}
\frac{1}{\pi(x) \#\cR \#\cS} \sum_{\substack{r \in \cR, s\in \cS \\ \Delta(rs) \ne 0}}
\pi_{E(rs)}(\alpha,\beta; x)& - \mu_{\tt ST}(\alpha,\beta) \\
& \ll \frac{ T^4x^{-1/4}  + T^3x^{1/4}\log x}{ \#\cR \#\cS}.
\end{split}
\end{equation*}
\end{theorem}

Thus, for $\cR = \cS = \cF(T)$, recalling~\eqref{eq:Farey}, we derive the
following multiplicative  analogue of~\cite[Theorem~6]{ShaShp}. In turn, we also note
that~\cite[Theorem~6]{ShaShp} can  be extended to sum sets of arbitrary
sets $\cR, \cS \subseteq \cF(T)$.

\begin{corollary}
\label{cor:S-T Farey Prod}
Suppose that the polynomials $f(Z), g(Z) \in \Z[Z]$
satisfy~\eqref{eq:Nondeg}, and
 for some sufficiently small  $\eps > 0$ we have
$$
  T \ge x^{1/4+\eps}.
$$
Then, uniformly over  $[\alpha,\beta] \subseteq [0,\pi]$, we have
$$
 \frac{1}{ \pi(x) (\#\cF(T))^2}\sum_{\substack{r,s \in \cF(T) \\ \Delta(rs) \ne 0}}
\pi_{E(rs)}(\alpha,\beta; x)
 = \mu_{\tt ST}(\alpha,\beta) +O\(x^{-\eps}\log x  \).
$$
\end{corollary}

\section{Preliminaries}
\label{Preliminary}

\subsection{Primes of good reduction}
\label{sec:not}

We start with the observation that the condition~\eqref{eq:Nondeg}
(over any field $\K$ of characteristic $p >3$)
implies that $\Delta(Z)\in \K[Z]$ is not a constant polynomial. Indeed, if $\Delta(Z) = c\ne 0$ for some $c \in \K$, then $f(Z)$ and $g(Z)$ have no common roots.  Since $j(Z)$ is not constant, both $f$ and $g$ are also not constant. Now,  considering the derivative $\Delta(Z)' = 0$, we easily see that $f$ and $g$ must have common roots, which leads to a contradiction.

For  $t \in \Q$, let $N(t)$ denote the conductor of the  specialisation of
$E(Z)$ at $Z = t$. We always consider rational numbers in
the form of irreducible fraction.

Note that for $t\in \Q$, the discriminant $\Delta(t)$ may be a rational number. However, we know that the elliptic curve $E(t)$ has good reduction at prime $p$ if and only if $p$ does not divide both the numerator and denominator of $\Delta(t)$;
see~\cite[Chapter~VII, Proposition~5.1(a)]{Silv}.
So, we can say that for any prime $p$, $p \nmid N(t)$ (that is, $E(t)$ has good reduction at $p$) if and only if $\Delta(t)\not\equiv 0 \pmod p$ (certainly, it first requires that $p$ does not divide the denominator of $\Delta(t)$).

\subsection{Preparations for distribution of angles}
\label{sec:preparation}

Given an angle $\vartheta \in [0,\pi]$ and an integer $n\ge 1$, we define the function
\begin{equation}
\label{eq:sym_n}
\sym_n (\vartheta) =\frac{\sin \((n+1)\vartheta\) }{\sin \vartheta}.
\end{equation}
Note that for any $n \ge 2$, we have
 $$
 \sym_n (\vartheta) = \frac{\sin n \vartheta}{\sin \vartheta} \cos \vartheta + \cos n\vartheta =
 \sym_{n-1} (\vartheta) \cos \vartheta + \cos n\vartheta,
 $$
 which implies (via a simple inductive argument) that
\begin{equation}
\label{eq:symn}
\sym_n (\vartheta) \ll n.
\end{equation}

The following result is based on the ideas of Niederreiter~\cite{Nied},
and has been used implicitly in a number of works (see, for example,~\cite{ShaShp}).
It is also explicitly given in~\cite[Corollary~3.2]{dlBSSV}.

\begin{lemma}
\label{lem:ST Discrep}
Given $m$ arbitrary angles $\psi_1,\ldots,\psi_m \in [0, \pi]$ (not necessarily distinct), assume that
for every integer $n \ge 1$ we have
$$
 \left| \sum_{i=1}^{m} \sym_n(\psi_i) \right| \le n \sigma
$$
for some real $\sigma \ge 2$.
Then, uniformly over  $[\alpha,\beta] \subseteq [0,\pi]$, we have
$$
\# \{\psi_i\in [\alpha,\beta]~:~1\le i \le m\} =  \mu_{\tt ST}(\alpha,\beta) m
+ O\(\sqrt{m\sigma}\).
$$
\end{lemma}

 \subsection{Exponential sums with ratios}

For an  integer  $m$, we denote
$$
\em(z) = \exp(2 \pi i z/m).
$$
The following result is essentially given in~\cite[Lemma~7]{Shp_LinEq}.

\begin{lemma}
\label{lem:Sum Double}
 Let $T < p$ for a prime $p$ and let $\cI(A,T)$ and $\cJ(B,T)$ be
two intervals of the form~\eqref{eq:interv}.
Let $\cW\subseteq \cI(A,T)\times\cJ(B,T)$
be an arbitrary convex subset.   Then
$$
\max_{a \in \F_p^*}
 \left|  \sum_{\substack{(u,v) \in \cW \cap \Z^2 \\ \gcd(v,p)=1}} \ep(au/v) \right| \le   T^{1/2}p^{1/2+o(1)}.
$$
\end{lemma}

The following bound that holds for any prime $p$ and integer $T\ge 2$,
\begin{equation}
\label{eq:ExpSum FN}
\max_{a \in \F_p^*} \left| \sum_{u/v \in \cF(T)} \ep(au/v)\right| \le T
(Tp)^{o(1)} + T^2/p,
\end{equation}
is  actually  a corrected form of~\cite[Lemma~15]{ShaShp} where also
the term $T^2/p$ has to be added. Indeed in the proof
of~\cite[Lemma~15]{ShaShp}, we omitted the case $\gcd(v,p)\ne 1$, which makes contributions at most $T(T/p+1)$, and so one should also add $T^8/p^4$ to the error
term in~\cite[Lemma~16]{ShaShp}. Fortunately, this change does not affect
the result in~\cite[Theorem~5]{ShaShp}, whose proof relies on~\cite[Lemma~16]{ShaShp}.

Now, we need to generalize the bound~\eqref{eq:ExpSum FN}.
 We first note that the proof in~\cite[Lemma~3]{Shp0} (the condition $L_x < m$ there can be deleted) actually gives the following estimate.

\begin{lemma}
\label{lem:expsum Farey1}
Let $U, V,W$ be arbitrary positive integers with $V < W$. Assume that for each integer $v$ we are given two integers $L_v, U_v $ with $0\le L_v < U_v \le U$. Then for any integer $a \not\equiv 0$ {\rm (mod $m$)}, we have
$$
 \left|
 \sum_{\substack{v=V\\ \gcd(v,m)=1}}^W \sum_{u=L_v+1}^{U_v} \em(au/v)\right| \le (U+W)
(Wm)^{o(1)}.
$$
\end{lemma}

We recall the definitions~\eqref{eq:set Z} and ~\eqref{eq:set Z*}  of the sets $\cZ(A,B,T)$
and $ \cZ^*(A,B,T)$.

Applying similar arguments as in the proof of~\cite[Lemma~15]{ShaShp}
and using Lemma~\ref{lem:expsum Farey1} (taking $m=p$), we derive:

\begin{lemma}
\label{lem:expsum Farey2}
 Let $A,B \le T$ and let $\cI(A,T)$ and $\cJ(B,T)$ be
two intervals of the form~\eqref{eq:interv}.    Then for any prime $p$,
we have
$$
\max_{a \in \F_p^*} \left|
\sum_{(u,v) \in \cZ(A,B,T)}
  \ep(au/v)\right| \le T
(Tp)^{o(1)} + T^2/p,
$$
and
$$
\max_{a \in \F_p^*} \left|
\sum_{(u,v) \in \cZ^*(A,B,T)}
\ep(au/v)\right| \le T
(Tp)^{o(1)} + T^2/p.
$$
\end{lemma}

\subsection{Bounds on some single sums}
\label{sec:Sing}

Michel~\cite[Proposition~1.1]{Mich} gives the following bound for the sum of the function
$\sym_n (\vartheta)$, given by~\eqref{eq:sym_n} twisted by
additive characters,
we refer to~\cite{IwKow} for a background on characters.

\begin{lemma}
\label{lem:Mich bound1}
If the polynomials $f(Z), g(Z) \in \Z[Z]$
satisfy~\eqref{eq:Nondeg}, then for any prime $p$, we have
$$
\sum_{\substack{w \in \F_p\\
\Delta(w) \ne 0}}
\sym_n\(\psi_p(E(w))\) \ep\(mw\) \ll
np^{1/2},
$$
uniformly over  all
integers $m\ge 0$ and $n\ge 1$.
\end{lemma}

We also need the following analogue of Lemma~\ref{lem:Mich bound1},
which is given in~\cite[Lemma~3.4]{dlBSSV}.

\begin{lemma}
\label{lem:Mich bound2}
If the polynomials $f(Z), g(Z) \in \Z[Z]$
satisfy~\eqref{eq:Nondeg}, then for any prime $p$ and any  multiplicative character $\chi$ of $\F_p$,  we have
$$
 \sum_{\substack{w \in \F_p\\
\Delta(w) \ne 0}}
\sym_n\(\psi_p(E(w))\) \chi\(w\)
\ll np^{1/2},
$$
uniformly over  all integer $n\ge 1$.
\end{lemma}

\subsection{Frobenius angles over ratios}
\label{sec:Farey}

In this section, we estimate several sums of the function $\sym_n (\vartheta)$ when $\vartheta$ runs through Frobenius angles over ratios.
We start with a general result.

\begin{lemma}
\label{lem:sym conv}
 Let $T < p$ for a prime $p$ and let
$\cI(A,T)$ and $\cJ(B,T)$ be
two intervals of the form~\eqref{eq:interv}.
Let $\cW\subseteq \cI(A,T)\times\cJ(B,T)$
be an arbitrary convex subset.   Then
$$
 \sum_{\substack{(u,v) \in \cW \cap \Z^2\\
\Delta(u/v) \not \equiv 0 \pmod p}}
\sym_n(\psi_p(E(u/v)))  \ll nT^{1/2} p^{1 + o(1)},
$$
uniformly over all integers $n\ge 1$.
\end{lemma}

\begin{proof}
Using the orthogonality of the exponential function, we write
\begin{equation*}
\begin{split}
&\sum_{\substack{(u,v) \in \cW \cap \Z^2\\
\Delta(u/v) \not \equiv 0 \pmod p}}
\sym_n(\psi_p(E(u/v)))\\
& \quad =\sum_{\substack{w \in \F_p\\
\Delta(w) \ne 0 }}
\sym_n(\psi_p(E(w)))
\sum_{\substack{(u,v) \in \cW \cap \Z^2 \\ \gcd(v,p)=1 }}
\frac{1}{p} \sum_{m=0}^{p-1}\ep(m(w - u/v)) + O(nT),
\end{split}
\end{equation*}
where the term $O(nT)$ comes from the case $\gcd(v,p)\ne 1$ by using~\eqref{eq:symn}. Note that since  $T < p$, at most one $v$ is divisible by $p$.
Now, changing the order of summation we obtain:
\begin{equation*}
\begin{split}
\sum_{\substack{ (u,v) \in \cW \cap \Z^2\\
\Delta(u/v) \not \equiv 0 \pmod p}}&
\sym_n(\psi_p(E(u/v)))\\
&  = \frac{1}{p} \sum_{m=0}^{p-1}
  \sum_{\substack{w \in \F_p\\
\Delta(w) \ne 0 }}
\sym_n(\psi_p(E(w)))\ep(mw) \\
& \qquad \qquad \qquad \qquad \quad
\sum_{\substack{(u,v) \in \cW \cap \Z^2 \\ \gcd(v,p)=1}}  \ep(-mu/v) + O(nT).
\end{split}
\end{equation*}
Combining Lemma~\ref{lem:Mich bound1} with Lemma~\ref{lem:Sum Double}, we have
\begin{equation*}
\begin{split}
&\sum_{\substack{(u,v) \in \cW \cap \Z^2\\
\Delta(u/v) \not \equiv 0 \pmod p}}
\sym_n(\psi_p(E(u/v)))\\
&\qquad\quad \ll n p^{-1/2} \sum_{m=0}^{p-1}
\left|\sum_{\substack{(u,v) \in \cW \cap \Z^2 \\ \gcd(v,p)=1}}  \ep(-mu/v)  \right| + nT\\
&\qquad\quad \ll n p^{-1/2} \( \sN(\cW) + T^{1/2}p^{3/2+o(1)} \) + nT \\
&\qquad\quad \ll np^{-1/2}T^2 + nT^{1/2} p^{1 + o(1)} + nT \\
&\qquad\quad \ll nT^{1/2} p^{1 + o(1)},
\end{split}
\end{equation*}
which concludes the proof.
\end{proof}

We  again recall the  definitions~\eqref{eq:set Z} and ~\eqref{eq:set Z*}
of the sets $\cZ(A,B,T)$ and $ \cZ^*(A,B,T)$.
The following  results gives an improvement upon the estimate in
Lemma~\ref{lem:sym conv} when $\cW =\cI(A,T)\times\cJ(B,T)$.
Note that here we do not need the condition $p> T$ any more.
The proof is fully analogous to that of Lemma~\ref{lem:sym conv},
except that instead of Lemma~\ref{lem:Sum Double} one has to
apply Lemma~\ref{lem:expsum Farey2}:

\begin{lemma}
\label{lem:sym Farey1}
 Let $A,B \le T$,  and
let  $\cY$ be one of the sets $\cZ(A,B,T)$
 or $\cZ^*(A,B,T)$.
 If the polynomials $f(Z), g(Z) \in \Z[Z]$
satisfy~\eqref{eq:Nondeg}, then for any prime $p$, uniformly over all integers $n\ge 1$
 we have
 $$
 \sum_{\substack{(u,v) \in \cY\\
\Delta(u/v) \not \equiv 0 \pmod p}}
\sym_n(\psi_p(E(u/v)))  \ll
nT^2p^{-1/2} + nT^{1+o(1)} p^{1/2 + o(1)}.
$$
\end{lemma}

Similarly, we get the following over the product set $\cF(T)\times \cF(T)$.

\begin{lemma}
\label{lem:sym Farey2}
If the polynomials $f(Z), g(Z) \in \Z[Z]$
satisfy~\eqref{eq:Nondeg},
then for any prime $p$, integer $T\ge 2$ and sets $\cR, \cS \subseteq \cF(T)$,
 we have
$$
\sum_{\substack{r \in \cR, s\in \cS\\
\Delta(rs) \not \equiv 0 \pmod p}}
\sym_n(\psi_p(E(rs)))  \ll
n \(T^4p^{-1/2} + T^2p^{1/2}(\log p)^2 \),
$$
uniformly over all integers $n\ge 1$.
\end{lemma}

\begin{proof} We fix  a primitive multiplicative character $\chi$ of $\F_p$.
Using the orthogonality of the character function, we write
$$
\sum_{\substack{r \in \cR, s\in \cS\\
\Delta(rs) \not \equiv 0 \pmod p}}
\sym_n(\psi_p(E(rs)))
  = \Omega_1 + \Omega_2,
$$
where
\begin{equation*}
\begin{split}
 \Omega_1&
  = \sum_{\substack{r \in \cR, s\in \cS\\ \textrm{$p\mid r$, or $p\mid s$} \\
\Delta(rs) \not \equiv 0 \pmod p}}
\sym_n(\psi_p(E(rs))), \\
  \Omega_2&= \sum_{\substack{w \in \F_p^*\\
\Delta(w) \ne 0 }}
\sym_n(\psi_p(E(w)))\\
& \qquad \qquad \qquad
\sum_{\substack{u_1/v_1\in \cR, \,\gcd(u_1v_1,p)=1 \\ u_2/v_2\in \cS, \,\gcd(u_2v_2,p)=1}}
\frac{1}{p-1} \sum_{m=1}^{p-1}\chi^m(wv_1v_2/(u_1u_2)),
\end{split}
\end{equation*}
where $p\mid t$  means that $p$ divides the denominator or numerator of  a rational
number  $t\ne 0$.

Using~\eqref{eq:symn}, we directly have
\begin{equation}
\label{eq:Omega1}
\Omega_1 \ll nT^3(T/p +1).
\end{equation}
Changing the order of summation, we obtain:
\begin{equation*}
\begin{split}
&\Omega_2= \frac{1}{p-1} \sum_{m=1}^{p-1}
  \sum_{\substack{w \in \F_p^*\\
\Delta(w) \ne 0 }}
\sym_n(\psi_p(E(w)))\chi^m(w) \\
& \qquad \qquad\qquad \qquad
\sum_{\substack{u_1/v_1\in \cR \\ \gcd(u_1v_1,p)=1}}  \chi^m(v_1/u_1)
\sum_{\substack{u_2/v_2\in \cS\\ \gcd(u_2v_2,p)=1}}  \chi^m(v_2/u_2) .
\end{split}
\end{equation*}
Using Lemma~\ref{lem:Mich bound2}, we have
$$
\Omega_2 \ll n p^{-1/2} \sum_{m=1}^{p-1}
\left|\sum_{\substack{u_1/v_1\in \cR \\ \gcd(u_1v_1,p)=1}}  \chi^m(v_1/u_1)  \right|
\left|\sum_{\substack{u_2/v_2\in \cS \\ \gcd(u_2v_2,p)=1}}  \chi^m(v_2/u_2) \right|.
$$
By the Cauchy inequality, we get
\begin{equation}
\begin{split}
\label{eq:Cauchy}
\Omega_2^2   \ll \frac{n^2}{p} \sum_{m=1}^{p-1} &
\left|\sum_{\substack{u_1/v_1\in \cR \\ \gcd(u_1v_1,p)=1}}  \chi^m(v_1/u_1)  \right|^2 \\
&\qquad  \qquad  \qquad  \sum_{m=1}^{p-1} \left|\sum_{\substack{u_2/v_2\in \cS \\ \gcd(u_2v_2,p)=1}}  \chi^m(v_2/u_2) \right|^2.
\end{split}
\end{equation}

Note that for any subset $\cQ \subseteq \cF(T)$, by
the orthogonality of characters, we have
\begin{align*}
\sum_{m=1}^{p-1}
\left|\sum_{\substack{u/v\in \cQ\\ \gcd(uv,p)=1}}  \chi^m(v/u) \right|^2
   =\sum_{m=1}^{p-1}&
\sum_{\substack{u_1/v_1, u_2/v_2\in \cQ \\ \gcd(u_1u_2v_1v_2,p)=1}}  \chi^m\(u_2v_1/(u_1v_2)\)   \\
 =
\sum_{\substack{u_1/v_1, u_2/v_2\in\cQ \\ \gcd(u_1u_2v_1v_2,p)=1}} &  \sum_{m=1}^{p-1} \chi^m\(u_2v_1/(u_1v_2)\) = (p-1)W,
\end{align*}
where $W$ is number of solutions to the congruence
$$
u_1/v_1 \equiv  u_2/v_2 \pmod p, \quad u_1/v_1, u_2/v_2\in\cQ, \  \gcd(u_1u_2v_1v_2,p)=1.
$$
Extending the range of variables to the whole interval $[1,T]$, and using~\cite[Lemma~14]{ShaShp} (in a slightly more precise form with $(\log p)^2$ instead of $p^{o(1)}$
given in the proof of ~\cite[Lemma~14]{ShaShp}), which in turn is essentially a version of
 a result of Ayyad, Cochrane and
Zheng~\cite[Theorem~2]{ACZ}, we obtain $W \ll  T^4/p   +  T^2(\log p)^2$.
Hence
$$
\sum_{m=1}^{p-1}
\left|\sum_{\substack{u/v\in \cQ\\ \gcd(uv,p)=1}}  \chi^m(v/u) \right|^2
\ll  T^4   +  T^2p(\log p)^2,
$$
and recalling~\eqref{eq:Cauchy} we derive
\begin{equation}
\label{eq:Omega2}
\Omega_2 \ll nT^4p^{-1/2} + nT^2p^{1/2} (\log p)^2.
\end{equation}
Since
$$
T^4p^{-1/2} + T^2p^{1/2} (\log p)^2 \ge nT^3,
$$
the bounds~\eqref{eq:Omega1} and~\eqref{eq:Omega2} imply the desired result.
\end{proof}

\section{Proofs of Main Results}

\subsection{Proof of Theorem~\ref{thm:S-T conv}}

For any prime $p$, let
$$
\cW_p = \{ (u,v) \in \cW \cap \Z^2~:~\Delta(u/v) \not\equiv 0 \pmod p \}.
$$
We denote by $M_p(\alpha,\beta;\cW)$ the number of pairs $(u,v)\in \cW_p$ such that  $\psi_p(E(u/v)) \in [\alpha,\beta]$.

It follows from Lemma~\ref{lem:sym conv} that for prime $p>T$, we have
$$
  \left| \sum_{(u,v) \in \cW_p} \sym_n(\psi_p(E(u/v))) \right|
  \ll nT^{1/2} p^{1 + o(1)}.
$$
So, using Lemma~\ref{lem:ST Discrep}, we have
$$
M_p(\alpha,\beta;\cW) =  \mu_{\tt ST}( \alpha,\beta)  \# \cW_p
  + O\(\sqrt{T^{1/2} p^{1 + o(1)}\# \cW_p}\).
$$
Noticing for $p>T$,
$$
\sN(\cW) - T - T\deg \Delta \le \# \cW_p \le \sN(\cW)
$$
(where the term $-T$ comes from the exceptional case $p \mid v$),
we obtain
\begin{equation}
\begin{split}
\label{eq:M_pW1}
 M_p(\alpha,\beta;\cW)  & - \mu_{\tt ST}(\alpha,\beta)\sN(\cW)  \\
&\qquad \ll T +  T^{1/4}p^{1/2+o(1)}\sN(\cW)^{1/2}.
\end{split}
\end{equation}
For $p \le T$, we use the trivial bound
\begin{equation}
\label{eq:M_pW2}
 M_p(\alpha,\beta;\cW)   - \mu_{\tt ST}(\alpha,\beta)\sN(\cW)  \ll\sN(\cW)  .
\end{equation}

Besides, it is easy to see that
\begin{align*}
\sum_{\substack{(u, v) \in \cW \cap \Z^2 \\ \Delta(u/v) \ne 0}}
\pi_{E(u/v)}(\alpha,\beta; x)
& = \sum_{\substack{(u, v) \in \cW \cap \Z^2 \\ \Delta(u/v) \ne 0}}
\sum_{\substack{p\le x \\ \Delta(u/v) \not\equiv 0  \pmod p \\ \psi_p(E(u/v)) \in [\alpha,\beta]}} 1 \\
& = \sum_{p\le x} \sum_{\substack{(u, v) \in \cW \cap \Z^2  \\ \Delta(u/v) \not\equiv 0  \pmod p \\ \psi_p(E(u/v)) \in [\alpha,\beta]}} 1
 = \sum_{p\le x} M_p(\alpha,\beta;\cW).
\end{align*}

Thus, applying~\eqref{eq:M_pW1} and~\eqref{eq:M_pW2} we deduce that
\begin{align*}
& \sum_{\substack{(u, v) \in \cW \cap \Z^2 \\ \Delta(u/v) \ne 0}}
\pi_{E(u/v)}(\alpha,\beta; x) - \sum_{p\le x} \mu_{\tt ST}(\alpha,\beta)\sN(\cW)  \\
& \qquad = \sum_{p\le x} \( M_p(\alpha,\beta;\cW) - \mu_{\tt ST}(\alpha,\beta)\sN(\cW) \) \\
& \qquad  \ll T\sN(\cW)  + \sum_{T<p\le x}\( T +  T^{1/4}p^{1/2+o(1)}\sN(\cW) ^{1/2} \)\\
& \qquad \ll T\sN(\cW)  + T\pi(x) + T^{1/4}x^{1/2+o(1)}\sN(\cW)^{1/2} \pi(x).
\end{align*}
Now, the desired result follows from dividing both sides by $\pi(x)\sN(\cW)$.

\subsection{Proof of Theorem~\ref{thm:S-T Farey}}
Since the proofs of the two cases are similar, we only present a proof for one case.

For any prime $p$, let
$$
\cZ_p = \{ (u,v) \in \cZ(A,B,T)~:~\Delta(u/v) \not\equiv 0 \pmod p \}.
$$
We denote by $M_p(\alpha,\beta;\cZ)$ the number of pairs $(u,v)\in \cZ_p$ such that  $\psi_p(E(u/v)) \in [\alpha,\beta]$.

It follows from Lemma~\ref{lem:sym Farey1} that
$$
  \left| \sum_{(u,v) \in \cZ_p} \sym_n(\psi_p(E(u/v))) \right|
  \ll nT^2p^{-1/2} + nT^{1+o(1)} p^{1/2 + o(1)}.
$$
So, using Lemma~\ref{lem:ST Discrep}, we have
\begin{equation*}
\begin{split}
M_p(\alpha,\beta;\cZ) = & \mu_{\tt ST}( \alpha,\beta)  \# \cZ_p \\
&\quad  + O\(\sqrt{\(T^2p^{-1/2} + T^{1+o(1)} p^{1/2 + o(1)} \)\# \cZ_p}\).
\end{split}
\end{equation*}
Noticing that
$$
 \# \cZ(A,B,T)  - T(T/p + 1) - T(T/p + 1)\deg \Delta \le \# \cZ_p \le \# \cZ(A,B,T)
$$
(where the term $- T(T/p + 1)$ comes from the exceptional case $p \mid v$),
we obtain
\begin{equation}
\begin{split}
\label{eq:M_pZ}
 M_p(\alpha,\beta;\cZ)  & - \mu_{\tt ST}(\alpha,\beta)\# \cZ(A,B,T) \\
&\qquad \ll T^2p^{-1/4} +  T^{3/2+o(1)}p^{1/4+o(1)}.
\end{split}
\end{equation}
In addition, as in the above it is easy to see that
$$
\sum_{\substack{(u,v)\in \cZ(A,B,T) \\ \Delta(u/v) \ne 0}}
\pi_{E(u/v)}(\alpha,\beta; x)
 = \sum_{p\le x} M_p(\alpha,\beta;\cZ).
$$

Thus, applying~\eqref{eq:M_pZ} we deduce that
\begin{align*}
 \sum_{\substack{(u,v)\in \cZ(A,B,T) \\ \Delta(u/v) \ne 0}} &
\pi_{E(u/v)}(\alpha,\beta; x) - \sum_{p\le x} \mu_{\tt ST}(\alpha,\beta)\# \cZ(A,B,T) \\
& \qquad  \ll \sum_{p\le x}\( T^2p^{-1/4} +  T^{3/2+o(1)}p^{1/4+o(1)}  \)\\
& \qquad \ll T^2x^{3/4}/\log x + T^{3/2+o(1)}x^{1/4+o(1)} \pi(x)\\
& \qquad \ll T^2x^{3/4}/\log x + T^{3/2+o(1)}x^{5/4+o(1)}.
\end{align*}
Then, the desired result follows easily from dividing both sides by $\pi(x)\# \cZ(A,B,T)$ and the fact that $\# \cZ(A,B,T) \asymp T^2$.

\subsection{Proof of Theorem~\ref{thm:S-T Farey Prod}}
Denote $\cT=\cR \times \cS$.
For any prime $p$, let
$$
\cT_p = \{ (r,s) \in \cT~:~\Delta(rs) \not\equiv 0 \pmod p \}.
$$
We denote by $M_p(\alpha,\beta;\cT)$ the number of pairs $(r,s)\in\cT_p$ such that  $\psi_p(E(rs)) \in [\alpha,\beta]$.

It follows from Lemma~\ref{lem:sym Farey2} that
$$
  \left| \sum_{(r,s) \in\cW_p} \sym_n(\psi_p(E(rs))) \right|
  \ll n \(T^4p^{-1/2} +  T^2p^{1/2}(\log p)^2 \).
$$
So, using Lemma~\ref{lem:ST Discrep}, we have
$$
M_p(\alpha,\beta;\cT) - \mu_{\tt ST}( \alpha,\beta)  \#\cT_p  \ll \sqrt{\(T^4p^{-1/2} + T^2p^{1/2}(\log p)^2 \)\#\cT_p}.
$$
Noticing that
$$
 \# \cR \# \cS  - 2T^3(T/p + 1) - T^3(T/p + 1)\deg \Delta \le \#\cT_p \le \# \cR \# \cS ,
$$
we obtain
\begin{equation}
\label{eq:M_pFF}
 M_p(\alpha,\beta;\cT)  - \mu_{\tt ST}(\alpha,\beta)\# \cR \# \cS  \ll T^4p^{-1/4} + T^3p^{1/4}\log p.
\end{equation}
Besides, as in the above we have
$$
\sum_{\substack{r\in \cR, s\in \cS \\ \Delta(rs) \ne 0}}
\pi_{E(rs)}(\alpha,\beta; x)
 = \sum_{p\le x} M_p(\alpha,\beta;\cT).
$$

Thus, applying~\eqref{eq:M_pFF} we deduce that
\begin{align*}
 \sum_{\substack{r\in \cR, s\in \cS \\ \Delta(rs) \ne 0}} &
\pi_{E(rs)}(\alpha,\beta; x) -  \mu_{\tt ST}(\alpha,\beta)  \# \cR \# \cS\pi(x)  \\
& \qquad  \ll \sum_{p\le x}\( T^4p^{-1/4}  + T^3p^{1/4}\log p  \)\\
& \qquad \ll T^4x^{3/4}/\log x  + T^3x^{5/4},
\end{align*}
and  the desired result now follows.

\section*{Acknowledgements}
This work was supported by the Australian
Research Council Grant DP130100237.


\begin{thebibliography}{9999}

\bibitem{ACZ} A.~Ayyad, T.~Cochrane and Z.~Zheng,
\textit{The congruence $x_1x_2 \equiv x_3x_4 \pmod p$, the equation
$x_1x_2 = x_3x_4$ and the mean values of character sums},
J.  Number Theory \textbf{59} (1996), 398-413.


\bibitem{Baier1} S.~Baier,
\textit{The Lang--Trotter conjecture on average},  J.\ Ramanujan
Math. Soc. \textbf{22} (2007), 299--314.

\bibitem{Baier2} S.~Baier,
\textit{A remark on the Lang--Trotter conjecture}, in:  New Directions in Value-Distribution Theory of Zeta and L-functions, R. Steuding and J. Steuding (eds.), Shaker Verlag, 2009, 11--18.


\bibitem{BaZha} S.~Baier and L.~Zhao,
\textit{The Sato--Tate conjecture on average for small angles},
Trans. Amer. Math. Soc. \textbf{361} (2009), 1811--1832.

\bibitem{BaSh}  W. D. Banks and I. E. Shparlinski,
\textit{Sato--Tate, cyclicity, and divisibility statistics
on average for elliptic curves of small height},
 Israel J. Math. \textbf{173} (2009), 253--277.

\bibitem{B-LGHT} T. Barnet-Lamb, D. Geraghty,
M. Harris and R. Taylor, \textit{A family of Calabi-Yau varieties
and potential automorphy II},  Publ. Res. Inst.
Math. Sci. \textbf{47} (2011), 29--98.

\bibitem{BeRob}
M. Beck and S. Robins, {\it Computing the continuous discretely\/}, Second Edition, Springer, New York, 2015.





\bibitem{dlBSSV}  R. de la Bret\`eche,
M. Sha, I. E. Shparlinski and J. F. Voloch, \textit{The Sato--Tate distribution in thin parametric families of elliptic curves}, preprint, 2015, available at \url{http://arxiv.org/abs/1509.03009}.

\bibitem{CHT}
L. Clozel, M. Harris and R. Taylor,
\textit{Automorphy for some $\ell$-adic lifts of automorphic mod $\ell$ Galois  representations},
 Pub. Math. IHES \textbf{108} (2008), 1--181.


\bibitem{Coj} A. C. Cojocaru,
\textit{Questions about the reductions modulo primes of an elliptic curve},
in: Proc. 7th Meeting of the Canadian
Number Theory Association (CRM Proceedings and Lecture Notes 36),  E. Goren and H. Kisilevsky (ed.), Amer. Math. Soc., 2004, 61--79.

\bibitem{CojDav} A. C. Cojocaru and C. David,
\textit{Frobenius fields for elliptic curves},
 Amer. J. Math. \textbf{130} (2008), 1535--1560.

\bibitem{CojHal} A. C.  Cojocaru and C. Hall,
\textit{Uniform results for Serre's theorem for elliptic curves},
 Int. Math. Res. Not. \textbf{2005} (2005), 3065--3080.


\bibitem{CojShp} A. C. Cojocaru and  I. E. Shparlinski,
\textit{Distribution of Farey fractions in residue classes
and Lang--Trotter conjectures on average},
 Proc. Amer. Math. Soc. \textbf{136} (2008), 1977--1986.


\bibitem{DavPapp} C. David and F. Pappalardi,
\textit{Average Frobenius distributions of elliptic
curves},  Int. Math. Res. Not. \textbf{1999} (1999), 165--183.





\bibitem{DavJ-U}
C. David and J. J. Urroz,
\textit{Square-free discriminants of Frobenius rings},
 Int. J. Number Theory \textbf{6} (2010),  1391--1412.


\bibitem{FoMu} {\'E}. Fouvry and M. R.  Murty,
\textit{On the distribution of supersingular
primes},  Canad. J. Math. \textbf{48} (1996),   81--104.




\bibitem{HW} G. H. Hardy and E. M. Wright, {\it An Introduction
to the Theory of Numbers\/}, Oxford University Press, Oxford, 1979.


\bibitem{HS-BT}
M. Harris, N. Shepherd-Barron and R. Taylor,
\textit{A family of Calabi-Yau varieties and potential automorphy},
 Ann. Math. \textbf{171} (2010), 779--813.



\bibitem{IwKow} H. Iwaniec and E. Kowalski,
 \textit{Analytic Number Theory}, Amer.  Math.  Soc.,
Providence, RI, 2004.






\bibitem{Lang}
S. Lang and H. Trotter, \textit{Frobenius Distributions in $\GL_2$-Extensions}, Lecture Notes in Math. \textbf{504}, Springer, Berlin, 1976.
%



\bibitem{Mich} P. Michel,
\textit{Rang moyen de familles de courbes elliptiques et
lois de Sato--Tate},
 Monatsh. Math. \textbf{120} (1995), 127--136.



\bibitem{Nied} H.~Niederreiter,
\textit{The distribution of values of Kloosterman sums}, Arch.
Math. \textbf{56} (1991),   270--277.





\bibitem{ShaShp} M. Sha and I. E. Shparlinski,
\textit{Lang--Trotter and Sato--Tate distributions
in single and double parametric families of elliptic curves},
Acta Arith. \textbf{170} (2015), 299--325.


\bibitem{Shp0}  I. E. Shparlinski,
\textit{Exponential sums with Farey fractions},
 Bull. Polish Acad. Sci. Math. \textbf{57}  (2009),
101--107.

\bibitem{Shp2} I. E. Shparlinski,
\textit{On the  Sato--Tate conjecture
on average for some families of  elliptic curves},
 Forum Math. \textbf{25} (2013), 647--664.

\bibitem{Shp3} I. E. Shparlinski,
\textit{On the Lang--Trotter and Sato--Tate conjectures on
average for polynomial families of elliptic curves},
 Michigan Math. J. \textbf{62} (2013), 491--505.

\bibitem{Shp4} I. E. Shparlinski,
\textit{Elliptic curves over finite fields:
Number theoretic and cryptographic aspects},
in: Advances in Applied Mathematics, Modeling, and Computational Science, R. Melnik and I. Kotsireas (eds.), Springer, 2013, 65--90.

\bibitem{Shp_LinEq}
I. E. Shparlinski, \textit{Linear congruences with ratios},
Proc. Amer. Math. Soc., to appear, available at \url{http://arxiv.org/abs/1503.03196}.


\bibitem{Silv} J. H. Silverman, {\it The arithmetic of elliptic
curves\/}, Springer, Berlin, 2009.
%

\bibitem{Taylor2008}
R. Taylor, \textit{Automorphy for some $\ell$-adic lifts of automorphic mod $\ell$ Galois representations II},  Pub. Math. IHES \textbf{108} (2008), 183--239.



\end{thebibliography}
\end{document}